\documentclass[12pt]{amsart}

\usepackage{geometry}
\usepackage{amsfonts}
\usepackage{amssymb}
\usepackage{calrsfs}
\usepackage[mathscr]{eucal}
\usepackage{amsmath}
\usepackage{amsthm}
\usepackage[all]{xy}
\usepackage[mathscr]{eucal}
\usepackage{graphicx}
\usepackage{longtable}
\usepackage{wrapfig}
\usepackage{algpseudocode}
\usepackage{algorithm}

\usepackage{color}
\definecolor{red-}{rgb}{1.0,0.0,0.0}
\definecolor{green-}{rgb}{0.0,0.7,0.0}
\definecolor{brown-}{rgb}{0.9,0.6,0.0}

\usepackage{hyperref}
\hypersetup{
    colorlinks=true,
    linkcolor=blue,
    citecolor=blue,
    filecolor=blue,
    urlcolor=blue
}

\newtheorem{defi}{Definition}[section]

\newtheorem{thm}[defi]{Theorem}

\newtheorem{cor}[defi]{Corollary}
\newtheorem{prop}[defi]{Proposition}
\newtheorem{lem}[defi]{Lemma}
\newtheorem{rem}[defi]{Remark}

\begin{document}

\title{Some Fano manifolds whose Hilbert polynomial\\  is totally reducible over $\mathbb Q$}

\author{Antonio Lanteri and Andrea Luigi Tironi}

\date{\today}

\address{Dipartimento di Matematica ``F. Enriques'',
Universit\`a degli Studi di Milano, Via C. Saldini, 50,  I-20133 Milano,
Italy}\email{antonio.lanteri@unimi.it}
\address{Departamento de Matem\'atica, Universidad de Concepci\'on, Barrio Universitario s/n, Casilla
160-C, Concepci\'on, Chile} \email{atironi@udec.cl}


\begin{abstract} Let $(X,L)$ be any Fano manifold polarized by a positive multiple of its
fundamental divisor $H$. The polynomial defining the Hilbert curve of $(X,L)$ boils down to being
the Hilbert polynomial of $(X,H)$, hence it is totally reducible over $\mathbb C$;
moreover, some of the linear factors appearing in the factorization have rational coefficients, e.g.
if $X$ has index $\geq 2$. It is natural to ask when the same happens for all linear factors.
Here the total reducibility over $\mathbb Q$ of the Hilbert polynomial is investigated for three
special kinds of Fano manifolds: Fano manifolds of large index, toric Fano manifolds of low dimension,
and Fano bundles of low coindex.
\end{abstract}

\dedicatory{In memory of Mauro Beltrametti}

\maketitle

\section*{Introduction}\label{newintro}

Let $P_{-K_X}(m):=\chi(m(-K_X))$ be the Hilbert
polynomial of a Fano manifold of dimension $n$ (with respect to the anticanonical polarization).
The fact that for $X=\mathbb P^n$ one has
$$P_{-K_{\mathbb P^n}}(z) = \frac{1}{n!} \prod_{i=1}^n \left(z + \frac{i}{n+1}\right)$$
gives rise to various questions concerning the roots of the Hilbert polynomial of a Fano manifold
in general.
The most celebrated one is the so-called {\it{narrow canonical strip hypothesis}}, claiming that
for any root $\alpha \in \mathbb C$ of $P_{-K_X}(z)$ its real part satisfies the conditions
\begin{equation}\label{strip}
-1 + \frac{1}{n+1} \leq \text{\rm{Re}}(\alpha) \leq - \frac{1}{n+1}.
\end{equation}
A large literature is devoted to it and some of its variants, starting with Golyshev's paper \cite{Go}
(e.g. see \cite{Man}, \cite{HHK}, \cite{BGM}). Another natural question, which is what we
deal with in this paper, is about the rationality of all roots,
as it happens in the case of $\mathbb P^n$. In other words,
\begin{equation}\label{question}
\text{\rm{for which Fano manifolds is the polynomial}}\  P_{-K_X}\ \text{\rm{totally reducible
over}}\  \mathbb Q\ ?
\end{equation}
\noindent
We emphasize that this question makes sense, since for every dimension $n \geq 2$ there are
Fano $n$-folds for which $P_{-K_X}$ is totally reducible over $\mathbb Q$
and others for which this is not true.
From our results, however, the total reducibility of $P_{-K_X}$ over $\mathbb Q$ does not seem
to be directly related to any geometric property of the Fano manifold $X$.
Our interest in \eqref{question} stems from the study of Hilbert curves
of polarized manifolds, a notion introduced in \cite{BLS} and further studied in \cite{L1}, \cite{L2}, \cite{LT}.
The Hilbert curve of a polarized manifold $(X, L)$ is the complex affine plane curve
 $\Gamma_{(X,L)} \subset \mathbb A^2_{\mathbb C}$ of degree $n=\dim(X)$, defined
by the Hilbert-like polynomial $p_{(X,L)}(x, y) :=\chi(xK_X + yL)$, where $x$ and $y$ are regarded as 
complex variables \cite[Section 2]{BLS}.
Specialize the setting to the following case:
\smallskip
\newline
\noindent $(\diamond)$\quad $X$ is a Fano $n$-fold, $n\geq 2$, of index $\iota_X$, and $L:= rH$, where $H$ is
the fundamental divisor (i.e. $-K_X = \iota_X H$) and $r$ is a positive integer.

\medskip
For $(X,L)$ as in ($\diamond$) we simply write $p$ for $p_{(X,L)}$ and $\Gamma$ for
$\Gamma_{(X,L)}$. In \cite[Lemma 3.1]{LT} the following fact is proven.\medskip

\begin{prop}\label{basic}
For any pair $(X,L)$ as in $(\diamond)$,
$$p(x,y) = R(x,y) \cdot \prod_{j=1}^{\iota_X-1} (ry-\iota_X x+j),$$
with $R(x,y) =   \sum_{j=0}^{c_X} a_j (ry-\iota_X x)^j$,
$c_X=n+1-\iota_X$ being the coindex of $X$, and $(a_0, \dots, a_{c_X})$ the solution of the linear system

\begin{equation}\label{linear system}
U\cdot
\left[ \begin{array}{ c }
a_0 \\
a_1 \\
\vdots \\
a_{c_X} \\
\end{array}\right]=
\left[ \begin{array}{ c }
\frac{h^0(\mathcal{O}_X)}{\delta(0)} \\
\frac{h^0(H)}{\delta(1)} \\
\vdots \\
\frac{h^0(c_X H)}{\delta(c_X)} \\
\end{array}\right]\ ,
\end{equation}
where $U$ is the Vandermonde matrix
\begin{equation}\label{U}
U:=\left( \begin{array}{ c c c c c }
1& 0 & \cdots & 0 & 0 \\
1& 1 & \cdots & 1 & 1  \\
1& 2 & \cdots & 2^{c_X-1} & 2^{c_X}  \\
\vdots&\vdots&  &\vdots &\vdots \\
1& c_X & \cdots & (c_X)^{c_X-1} & (c_X)^{c_X}  \\
\end{array}\right)\ ,
\end{equation}
and $\delta$ is the function defined by $\delta(t)=\frac{(t +\iota_X-1) !}{t !}$
for every $t \in \mathbb Z_{\geq 0}$.
\end{prop}

In particular, the special setting in ($\diamond$)
allows to convert $p(x,y)$ into a polynomial in a single complex variable,
which is, of course, totally reducible over $\mathbb C$  (hence
$p(x,y)$ is the product of $n$ polynomials of degree $1$ in $\mathbb C [x,y]$).
In fact,
\begin{equation}\label{p}
p(x,y) = P(z):= \left( \sum_{j=0}^{c_X} a_jz^j \right) \cdot \prod_{j=1}^{\iota_X-1} (z+j) = a \prod_{i=1}^n (z-\alpha_i),
\end{equation}
where $z: = ry-\iota_X x$, $a:=a_{c_X}$ and $\alpha_i \in \mathbb C$.
As a consequence,
$\Gamma$ is always reducible
into $n$ lines of $\mathbb A^2_{\mathbb C}$, which are parallel each other, with slope $\frac{\iota_X}{r}$.
Moreover, $\iota_X-1$ of them are defined over $\mathbb Q$
and evenly spaced.
As to the polynomial $P$ appearing in \eqref{p} note that
\begin{equation}\label{P=P_H}
P(z)=P_H(z)=\chi(zH)
\end{equation}
 is precisely the Hilbert polynomial of $(X,H)$.
In particular it is a numerical polynomial.
Because $H = \frac{1}{\iota_X}(-K_X)$, we have $P_H(z)= P_{-K_X}\left(\frac{z}{\iota_X}\right)$. So,
while looking at $H$ forces to rescaling the bounds in \eqref{strip}, in dealing with
\eqref{question} it is equivalent to work with $P$ instead of $P_{-K_X}$, as we will do in the following.

Of course, from the point of view of the Hilbert curve,
\eqref{question} is equivalent to asking when the points of $\Gamma$ defined over $\mathbb Q$
(or $\mathbb R$), denoted by
${\Gamma}_{\mathbb Q}$ (or ${\Gamma}_{\mathbb R}$), constitute $n$ lines of $\mathbb A^2_{\mathbb Q}$
(or $\mathbb A^2_{\mathbb R})$.
However, the role of $\Gamma$ is clearly less relevant in the setting ($\diamond$) than in the
case where $K_X$ and $L$ are linearly independent; hence we will confine
ourselves to consider $\Gamma$ in a few points, just to complete the picture.

\smallskip
In this paper we address question \eqref{question} when $X$ is either:
\begin{enumerate}
\item[a)]  a Fano $n$-fold of  index $\iota_X \geq n-2$,
\item[b)]  a toric Fano manifold of dimension $n \leq 4$, or
\item[c)] a Fano bundle of sufficiently high index.
\end{enumerate}

\noindent
With regard to a), we first recall that $P$
is totally reducible over $\mathbb Q$ for both $\mathbb P^n$ and $\mathbb Q^n$
(e.g. see \cite[Section 3]{LT} or Proposition \ref{projective space,quadric,del Pezzo} (i) and (ii)). So,
the first nontrivial situation to be investigated is that occurring for
$\iota_X=n-1$, i.e. essentially for pairs $(X,H)$ which are  del Pezzo manifolds \cite[pp.\ 44--45]{Fu}.
The answer is the following.

\begin{prop}\label{basic2}
Let $(X,L)$ be as in ($\diamond$) with $\iota_X \geq n-1 \geq 1$ and let $d:=H^n$. Then $P$
is totally reducible over $\mathbb Q$, except in the following cases:
\begin{enumerate}
\item[i)]  $X$ is a del Pezzo surface with $K_X^2 \leq 7$;
\item[ii)]  $(X,H)$ is a del Pezzo threefold or fourfold with $d \leq 5$;
\item[iii)] $(X,H)$  is a del Pezzo manifold of dimension $n \geq 5$ with
$d \leq 4$ (in particular, any del Pezzo manifold of dimension $\geq 7$);
\item[iv)] $(X,H)$ is a del Pezzo threefold of degree $7$, namely,
$X = \mathbb P\big(\mathcal O_{\mathbb P^2}(1) \oplus \mathcal O_{\mathbb P^2}(2)\big)$,
$H$ being the tautological line bundle.
\end{enumerate}
\end{prop}
As to $\Gamma$, in cases  {\rm{i)-iii)}}  ${\Gamma}_{\mathbb R}$
(as well as ${\Gamma}_{\mathbb Q}$) reduces to $n-2$ parallel lines;
in case {\rm{iv)}}, ${\Gamma}_{\mathbb R}$ consists of three lines, while
${\Gamma}_{\mathbb Q}$ reduces to the single line of equation $2x - ry - 1 = 0$.

\par
The above result follows from a complete description of $\Gamma_{\mathbb R}$ and $\Gamma_{\mathbb Q}$
we provide in Section \ref{Sec2} for del Pezzo manifolds (Theorem \ref{description} and Proposition \ref{2.4}),
relying on their classification \cite{Fu}
and on Proposition \ref{projective space,quadric,del Pezzo} (iii).
Furthermore we consider the case $\iota_X=n-2$.
Here, relying on several results on the classification of Mukai
$n$-folds \cite{IP}, we obtain a precise answer to \eqref{question} for any $n\geq 3$
(Theorem \ref{propo1} for $n=3$ and Theorem \ref{thm Mukai} for $n\geq 4$).

\smallskip
Concerning toric Fano manifolds, we have to mention that the Hilbert polynomial
of the anticanonical line bundle can be regarded as the Ehrhart polynomial of the associate
lattice polytope. Many properties of these numerical polynomials are known, e.\ g. concerning
the location of the roots and their distribution (see \cite[Sec.\ 4]{5authors} and
\cite{HHK}), although not related to problem \eqref{question}. So, in relation to b), in Section \ref{Sec3}
we provide the precise lists of pairs $(X,L)$ as in ($\diamond$) with $X$ a toric Fano $n$-fold ($n\leq 4$)
for which the Hilbert polynomial $P$ is totally reducible over $\mathbb Q$ (Corollary \ref{cor} for $n=3$
and Theorem \ref{prop1} for $n=4$; see also Proposition \ref{surface} for $n=2$). This is done by relying on \cite{Ba}.
A similar result could also be obtained for $n=5$ and $6$, but the lists one obtains in these two
cases are too long to be included in this paper.

\smallskip
Finally, case c) looks the most interesting. In Section \ref{Sec4}, assuming that $X$ is a Fano bundle
of index $\iota_X \geq \frac{n+1}{2}$, we show that
for a pair $(X,L)$ as in ($\diamond$) the polynomial $P$ is totally reducible
over $\mathbb Q$ unless $X = \mathbb P(\mathcal E)$, where
$\mathcal E = \mathcal O_{\mathbb P^m}(1)^{\oplus (m-1)} \oplus \mathcal O_{\mathbb P^m}(2)$ (Theorem \ref{final}).
In fact to investigate this exceptional case we use a program in Magma to deal with the system \eqref{linear system}.
This allows us to show that the factor
of $P(z)$ corresponding to the polynomial
$R(x,y)$ appearing in Proposition \ref{basic}
is not totally reducible over $\mathbb Q$ only for a limited number of
values of $m$ (we tested it only for $m\leq 150$), and we conjecture that the same is true for every $m\geq 2$.
There is a good reson for regarding $\frac{n+1}{2}$ as the appropriate threshold for answering
positively our problem in case c). For, let $X = Y \times \mathbb P^{m-2}$, where $(Y,h)$ is a
del Pezzo manifold of dimension $m \geq 3$.
Then $n:=\dim X = 2(m-1)$, and
$-K_X = (m-1)H$, where $H = h \boxtimes \mathcal O_{\mathbb P^{m-2}}(1)$. Therefore $X =\mathbb P_Y(\mathcal E)$,
where $\mathcal E=h^{\oplus (m-1)}$, is a Fano bundle of index $\iota_X = \frac{n}{2}$.
Since $\chi(kH)=\chi(kh) \chi\big(\mathcal O_{\mathbb P^{m-2}}(k)\big)$ for every $k$,
$P(z)$ is the product of the Hilbert polynomials of $Y$ and $\mathbb P^{m-2}$. The latter is
 totally reducible over $\mathbb Q$, while the former could not be totally reducible over $\mathbb Q$, in view
of Proposition \ref{basic2}.  Therefore, for any $(Y,h)$ appearing among the exceptions in Proposition \ref{basic2},
not even $P$ is totally reducible over $\mathbb Q$.
This allows one to produce plenty of Fano bundles
$X$ of dimension $n$ and index $\iota_X = \frac{n}{2}$, whose
Hilbert polynomials are not totally reducible over $\mathbb Q$.
Clearly, a similar conclusion can be obtained by working with the exceptions arising in case b).

To conclude, consider the Fano threefold
$X=\mathbb P\big(\mathcal O_{\mathbb P^2}(1) \oplus \mathcal O_{\mathbb P^2}(2)\big)$,
which is isomorphic to $\mathbb P^3$ blown-up at one point (sometimes denoted by $V_7$ in the literature).
We would like to stress that this Fano threefold appears in each of the three contexts a), b), c) that we have considered.
Moreover, it contributes to underline a significant role that arithmetic plays, in addition to geography,
in the study of Hilbert curves of polarized manifolds (see Remark \ref{analog-conjecture})
Indeed this paper grew out of this observation.

\section{Background material}\label{backgr}

 Varieties considered in this paper are defined over
 the field $\mathbb C$ of complex numbers. We use the standard notation
 and terminology
 from algebraic geometry. A manifold is any smooth projective variety.
 Tensor products of line bundles are denoted
 additively. The pullback of a vector bundle $\mathcal F$ on a manifold $X$
 by an embedding $Y \hookrightarrow X$ is simply denoted by $\mathcal F_Y$.
 We denote by $K_X$ the canonical bundle of a manifold $X$. A
\textit{Fano manifold} is a manifold $X$ such that $-K_X$ is an ample line bundle ($X$ is also called
 a {\textit{del Pezzo surface}} if $\dim X=2$). The
 \textit{index} $\iota_X$ of $X$ is defined as the greatest positive integer which divides $-K_X$ in
the Picard group Pic$(X)$ of $X$. We recall that $\iota_X \leq \dim X +1$, equality holding if and only if
$X = \mathbb P^n$, by the Kobayashi--Ochiai theorem.
By the \textit{coindex} of $X$ we simply mean the
nonnegative integer $c_X:= \dim X +1-\iota_X$. Moreover, we say that a polarized
manifold $(X,H)$ of dimension $n$ is a \textit{del Pezzo manifold}
(respectively a \textit{Mukai manifold}) if
$K_X+(n-1)H=\mathcal{O}_X$ (respectively
$K_X+(n-2)H=\mathcal{O}_X$).  We denote by $d:=H^n$ its degree.
Note that if $X$ is a Fano manifold of dimension $n$ and index $n-1$, then $\big(X, \frac{1}{n-1}(-K_X)\big)$
is a del Pezzo manifold. The converse is true except for the following pairs:
$\big(\mathbb P^2, \mathcal O_{\mathbb P^2}(3)\big)$,
$\big(\mathbb Q^2, \mathcal O_{\mathbb Q^2}(2)\big)$ for $n=2$ and
$\big(\mathbb P^3, \mathcal O_{\mathbb P^3}(2)\big)$ for $n=3$.

\smallskip
Let $X$ be any projective manifold of dimension $n$.
For any line bundle $D$ on $X$, consider the expression of the Euler--Poincar\'e  characteristic $\chi(D)$
provided by the Riemann--Roch-- Hirzebruch theorem
\begin{equation}\label{rrh}
\chi(D) = \frac{1}{n!} D^n - \frac{1}{2(n-1)!}K_X \cdot D^{n-1} + \text{\rm{terms of lower degree}}
\end{equation}
(a polynomial of degree $n$ in the Chern class of $D$, whose coefficients are polynomials with rational coefficients
in the Chern classes of $X$ \cite[Theorem 20.3.2]{Hi}).
In particular, for $(X,L)$ as in ($\diamond$), letting $D = xK_X+yL= zH$, and recalling that $P(z)=p(x,y)$, \eqref{rrh} says  that
$$P(z) = H^n \left(\frac{1}{n!} z^n + \frac{\iota_X}{2(n-1)!} z^{n-1} + \text{\rm{terms of lower degree}} \right).$$
Note that for any root $\alpha \in \mathbb C$ of $P$, also its conjugate
$\overline{\alpha}$ is a root, since $P \in \mathbb Q[z]$. Moreover,
taking into account the Serre involution $D \mapsto K_X-D$ acting on $\text{Num}(X)$, we have
\begin{equation}\label{Serre}
P(z)=\chi(zH) = (-1)^n \chi(K_X-zH) = (-1)^n P(-z-\iota_X).
\end{equation}
Hence for any root $\alpha \in \mathbb C$ of $P$, also $-\iota_X-\alpha$ is a root.
In particular, $P(\frac{-\iota_X}{2}) = 0$ if $n$ is odd.

\par

This discussion, combined with Proposition \ref{basic}, leads to the following result.

\begin{prop}\label{thm1}
Let $(X,L)$ be as in $(\diamond)$ and consider the polynomial $P(z)$ as in \eqref{p}, where $z=ry-\iota_Xx$.
Then the following properties hold:
\begin{enumerate}
\item[1)] the leading coefficient of $P(z)$ appearing in
\eqref{p} is $a=\frac{H^n}{n!}=\frac{1}{n!}\frac{(-K_X)^n}{(\iota_X)^n}=\frac{(-1)^n}{\Pi_{i=1}^{n}\alpha_i}$;
\item[2)] the subset $A:=\{\alpha_1,\dots,\alpha_n\}\subset\mathbb{C}$ of the roots of $P$
is symmetric with respect to the two orthogonal lines $\text{\rm{Im}}(z)=0$ and $\text{\rm{Re}}(z)= - \frac{\iota_X}{2}$,
hence with respect to the point $ - \frac{\iota_X}{2}$;
moreover, the integers in $A$, if any,
are exactly $\{1-\iota_X,\dots,-2,-1\}$ (which implies $\iota_X\geq 2$);
\item[3)] if either $n$ or $c_X$ is odd, then $z+\frac{\iota_X}{2}$ is a factor of $P(z)$;
moreover, if $c_X$ is odd and $n$ is even, then $\left(z+\frac{\iota_X}{2}\right)^2$ is a factor of $P(z)$. 
\end{enumerate}
\end{prop}

\begin{proof}
1) Recalling that $-K_X=\iota_XH$, we have $a:=a_{c_X}=\frac{(-1)^n}{\Pi_{i=1}^{n}\alpha_i}=\frac{H^n}{n!}=\frac{1}{n!}
\frac{(-K_X)^n}{(\iota_X)^n}$, because
$P(0)=\chi(\mathcal{O}_X)=1$ and
$$ P(z)=\chi (zH)= \frac{H^n}{n!}z^n+\dots\ , $$
by equation \eqref{rrh}, where dots stand for terms of lower degree in $z$.

\smallskip

\noindent 2) As we said, $P(\alpha_i)=0$ if and only if $P(\overline{\alpha_i})=0$,
since the coefficients of $P$ are real numbers. On the other hand, $P(\alpha_i)=0$ if and only if $P(-\alpha_i-\iota_X)=0$,
due to \eqref{Serre}. These facts imply the claimed symmetries of $A$.
Finally, recall that
$P(z)=\chi (zH)$. Let $\alpha_i\in A$ be an integer.
Since $\chi (\alpha_iH)=0$ and
\begin{equation*}
    \chi (\alpha_iH)=
    \begin{cases}
      h^0(\alpha_iH) & \mathrm{for}\ \alpha_i\geq 0 \ , \\
      (-1)^nh^0\big((-\iota_X-\alpha_i)H\big) & \mathrm{for}\ \alpha_i< 0 \ ,
    \end{cases}
\end{equation*}
we deduce that $\alpha_i < 0$ and $\iota_X+\alpha_i>0$, that is, $\alpha_i\in \{1-\iota_X,\dots,-2, -1\}$, which implies $\iota_X\geq 2$. On the other hand, from Proposition \ref{basic} we know that $\{1-\iota_X,\dots, -2, -1\}\subseteq A$.

\smallskip

\noindent 3) By 2)  the set $A$ is symmetric with respect to $-\frac{\iota_X}{2}$, hence, if $n$ is odd,
there exists $\alpha_i\in A$ such that $\alpha_i=-\alpha_i-\iota_X$, i.e. $\alpha_i=-\frac{\iota_X}{2}$. Thus $z+\frac{\iota_X}{2}$ is a factor of $P(z)$. Suppose now that $c_X$ is odd.
Since both $\{1-\iota_X,\dots,-2,-1\}$ and $A$ are
symmetric with respect to $-\frac{\iota_X}{2}$ and $R(x,y)$ in Proposition \ref{basic} has degree $c_X$,
we deduce that the factor of $P$ corresponding to $R(x,y)$ has a root $\alpha_{j}$
such that $\alpha_j=-\alpha_j-\iota_X$, i.e. $\alpha_j=-\frac{\iota_X}{2}$. Therefore, even in this case, $z+\frac{\iota_X}{2}$ is a factor of $P(z)$.
Finally, if $c_X$ is odd and $n$ is even, then also $\iota_X = n+1-c_X=:2p$ is even, hence $\iota_X-1= 2p-1$.
Thus the second polynomial factoring $P(z)$ in \eqref{p} has the form
$\prod_{j=1}^{\iota_x-1}(z+j) = \prod_{j=1}^{2p-1}(z+j)$ and $z+p = z+ \frac{\iota_X}{2}$ is one of its factors.
\end{proof}

Sometimes,
letting $w:=\frac{\iota_X}{2}+z$, i.e. $z =\frac{-\iota_X}{2} + w$, it is useful to look
at $P$ as a polynomial in $w$, namely $Q(w):=P\left(w-\frac{\iota_X}{2}\right)$; in this case,
\eqref{Serre} becomes $Q(w) = (-1)^n Q(-w)$.
As a consequence, the polynomial $Q$ contains only terms of degrees with the same parity as $n$. This is the advantage
of looking at $Q$ instead of $P$.

\smallskip
As an example, let us discuss here the case of surfaces, by using $Q$.
Let $(X,L)$ be as in ($\diamond$); since $X$ is a del Pezzo surface
we have $-K_X=\iota_X H$ and $\chi(\mathcal O_X)=1$. Then the Riemann--Roch theorem gives
$$P(z) = \chi(zH) =
1 + \frac{1}{2}zH(zH-K_X) =
\frac{1}{2} \left(\left[\left(z+\frac{\iota_X}{2}\right)H\right]^2 +\ 2 -\frac{1}{4} K_X^2  \right).$$
Hence, replacing $z+\frac{\iota_X}{2}$ with $w$, we get
\begin{equation}\label{Q}
Q(w)=\frac{1}{2}\left(\frac{K_X^2}{\iota_X^2}w^2 - \frac{K_X^2-8}{4}   \right).
\end{equation}
Thus $Q$, hence $P$, is totally  reducible  over $\mathbb Q$ if and only if $\frac{K_X^2-8}{K_X^2}$ is the square of a rational number.
Clearly, this implies  $K_X^2 \geq 8$ (total reducibility over $\mathbb R$), and then $K_X^2$
is either $8$ or $9$, in view of the well-known classification of del Pezzo surfaces (e.g. see \cite[(8.1)]{Fu}).
Moreover, from \eqref{Q} we obtain that
\begin{equation*} \notag
Q(w) = \begin{cases}
4 w^2 &  \mathrm{if}\ K_ X^2=8\ \mathrm{and}\ X=\mathbb  F_1,\\
w^2 &  \mathrm{if}\ K_ X^2=8\ \mathrm{and}\ X=\mathbb  P^1 \times \mathbb  P^1,\\
 \frac{1}{2}\left(w -  \frac{1}{2}\right) \left(w +  \frac{1}{2}\right)
& \mathrm{if}\ K_X^2=9,\ \mathrm{in\ which\ case}\ X=\mathbb P^2,
 \end{cases}
 \end{equation*}
recalling that in the three cases above $\iota_X= 1,2$ or $3$ respectively.
This, in turn, makes evident the total reducibility over $\mathbb Q$.
Coming back to $p(x,y)$
to include $\Gamma$ into the picture, this discussion can be summarized by the following result, which explains i) in Proposition \ref{basic2}.

\begin{prop}\label{surface}
Let $(X,L)$ be as in $(\diamond)$ with $n=2$. Then the following are equivalent:
\begin{enumerate}
\item[1)] $P$ is totally reducible over $\mathbb{Q};$
\item[2)] $P$  is totally reducible over $\mathbb{R};$
\item[3)] $K_X^2=8$ or $9$;
\item[4)] $(X,L)$ is one of the following polarized surfaces:
\begin{enumerate}
\item $\big(\mathbb{P}^2,\mathcal{O}_{\mathbb{P}^2}(r)\big)$, $\iota_X=3$ and $p_{(X,L)}(x,y)=\frac{1}{2}\left(ry-3x+1\right)\left(ry-3x+2\right)$;
\item $\big(\mathbb{P}^1\times\mathbb{P}^1,\mathcal{O}_{\mathbb{P}^1\times\mathbb{P}^1}(r,r)\big)$, $\iota_X=2$ and $p_{(X,L)}(x,y)=\left(ry-2x+1\right)^2$;
\item $(\mathbb{F}_1,-rK_{\mathbb{F}_1})$, $\iota_X=1$ and $p_{(X,L)}(x,y)=4\left(ry-x+\frac{1}{2}\right)^2$.
\end{enumerate}
\end{enumerate}
\end{prop}

\section{Fano manifolds of large index}\label{Sec2}

Given a Fano manifold $X$ of dimension $n$ and index $\iota_X\geq n-2$, it is known that there exists a smooth element
$Y \in |H|$. This is obvious for $\iota_X=n+1$ and $n$; it follows from Fujita's theory of del Pezzo manifolds \cite[$\S 8$]{Fu} for
$\iota_X=n-1$, and from a result of Mella \cite{Me} for $\iota_X=n-2$. Note that $-K_Y=(\iota_X-1)H_Y$ by adjunction. In particular, if $n\geq 3$ and $(X,H)$ is
a del Pezzo manifold, then $(Y,H_Y)$ is also a del Pezzo manifold, and similarly, if $n\geq 4$ and $(X,H)$ is
a Mukai manifold, then $(Y,H_Y)$ is a Mukai manifold too. A consequence of this fact is that for $\iota_X\geq n-2$ we can always apply
an inductive argument up to the surface case to compute $h^0(tH)$ for $t=1,\dots,c_X\leq 3$.
So, for $(X,L)$  as in ($\diamond$), this allows to make the polynomial $p(x,y)$ explicit. On the other hand, this polynomial
allows to recover $(X,L)$ provided that the polarization satisfies a mild arithmetic assumption. In fact, translating
\cite[Theorem 3.3 and Remark 3.5]{LT} in terms of $z=ry-\iota_X x$, we can easily obtain the following 
characterizations of pairs $(X,L)$ as in ($\diamond$) via Hilbert polynomials for $\iota_X\geq n-1$.
From now on we can assume $n \geq 3$,  in view of Proposition \ref{surface}.

\begin{prop}\label{projective space,quadric,del Pezzo}
Let $(X,L)$ be a polarized manifold of dimension $n\geq 3$.
\begin{enumerate}

\item[(i)]
$(X,L)=\big(\mathbb{P}^n,\mathcal{O}_{\mathbb{P}^n}(r)\big)$ for a positive integer $r$ coprime
with $n+1$ if and only if $H=\frac{1}{r}L$ is an ample line bundle and the Hilbert polynomial of $(X,H)$ is
$$P(z)= \frac{1}{n!}\prod_{i=1}^n(z+i).$$

\item[(ii)]
$(X,L)=\big(\mathbb{Q}^n,\mathcal{O}_{\mathbb{Q}^n}(r)\big)$ for a positive integer $r$ coprime
with $n$ if and only if $H=\frac{1}{r}L$ is an ample line bundle and the Hilbert polynomial of $(X,H)$ is
$$P(z)=  \frac{2}{n!}\left(z+\frac{n}{2}\right) \prod_{i=1}^{n-1}(z+i).$$

\item[(iii)]
$X$ is a Fano manifold of index $\iota_X = n-1$ and $L := \frac{r}{n-1}\left(-K_X\right)$
for a positive integer $r$ coprime with $n-1$
if and only if $H=\frac{1}{r}L$ is an ample line bundle such that the Hilbert polynomial of $(X,H)$ is
$$P(z)=\left(\frac{d}{n!}z^2 + \frac{(n-1)d}{n!}z + \frac{1}{(n-2)!}\right) \prod_{i=1}^{n-2}(z+i).$$
\end{enumerate}
\end{prop}

\begin{rem}\label{2.2}
{\em Notice that in Proposition \ref{projective space,quadric,del Pezzo}, the coprimality condition
is needed just to prove the ``if'' part, in all the three cases (i)--(iii).
However, in the following we do not care it, since we will use only the converse.
As to (iii), recall that the equality $\iota_X = n-1$ implies that $\big(X,\frac{1}{\iota_X}(-K_X)\big)$ is a del Pezzo manifold,
hence $d$ is its degree.
Moreover, for $n=3$, letting $d=8$ we see that the equation in (iii) coincides with that provided in (i)
by replacing $z$ with $2z$, which corresponds to taking
$(X,L)=\big(\mathbb P^3, \mathcal O_{\mathbb P^3}(2r)\big)$.
As a consequence, in (iii) we can suppose that $d\leq 7$ if $n=3$.
For the classification of del Pezzo manifolds we refer to \cite [(8.11)]{Fu}.}
\end{rem}

\smallskip

Let $(X,L)$ be as in ($\diamond$). Clearly the polynomial $P$ is totally reducible over $\mathbb Q$ in cases (i) and (ii)
of Proposition \ref{projective space,quadric,del Pezzo}.
Note that in case (ii), $P$ has a double root, namely $-\frac{n}{2}$, if and only if $n$ is even.
In case (iii) of
Proposition \ref{projective space,quadric,del Pezzo}, $P(z)$ factors into $(n-2)$ linear factors and the degree $2$ factor
\begin{equation}\label{trinomial}
\frac{d}{n!}\left(z^2 + (n-1)z +\frac{n(n-1)}{d}\right).
\end{equation}
Then $P$ is totally reducible over $\mathbb Q$ if and only if the same happens for this factor.
Leaving out the coefficient $\frac{d}{n!}$, the discriminant of the above trinomial is
\begin{equation}\label{Delta}
\Delta = \frac{n-1}{d}\big( (n-1)d-4n\big).
\end{equation}
Thus $P$ is totally reducible over $\mathbb Q$ if and only if $\Delta= k^2$, for some $k \in \mathbb Q$.
On the other hand, the total reducibility of $P$ over $\mathbb R$ is expressed by the condition $\Delta \geq 0$,
i.e.
$$
d \geq \frac{4n}{n-1}.
$$
This means that
$d\geq 6$ if $n=3$ or $4$ and $d\geq 5$  if $n \geq 5$.  Look at the del Pezzo manifold $(X,H)$.
 Fujita's classification \cite[(8.11)]{Fu} implies that $d \leq 4$ if $n \geq 7$, hence
$\Delta<0$ for $n \geq 7$, and therefore $P$ cannot be totally reducible over $\mathbb R$,
and a fortiori over $\mathbb Q$, in this case.
Consider also that $d\leq 8$ for $n=3$, $d\leq 6$ for $n=4$, and $d\leq 5$ if
$n=5, 6$ \cite[(8.11)]{Fu}. In particular, case $d=8$, which corresponds to
$(X,H) = \big(\mathbb P^3, \mathcal O_{\mathbb P^3}(2)\big)$,
fits into (i), as already observed in Remark \ref{2.2}.
In terms of coordinates $x$ and $y$, the
factor in \eqref{trinomial} defines the conic $G$ of equation
$\left( (n-1) \left(x-\frac{1}{2}\right) -ry\right)^2 - \frac{\Delta}{4} =0 $
and $\Gamma =\ell_1 + \dots + \ell_{n-2} + G$ consists of $G$ plus $(n-2)$ paralell lines $\ell_i$
whose equations are $ ry- (n-1)x+ i =0$ ($i=1, \dots, n-2$).

\smallskip
In conclusion, a case-by-case analysis leads to the following result, which includes a view on the
``geography'' of $\Gamma$, with obvious meaning of the symbols.

\begin{thm} \label{description}
Let $(X,L)$ be as in case ${\rm (iii)}$ of {\em
Proposition \ref{projective space,quadric,del Pezzo}} with $d\leq 7$.
Then $G_{\mathbb R} = \emptyset$, hence
$\Gamma_{\mathbb R} = \Gamma_{\mathbb Q}=\ell_1 + \dots + \ell_{n-2}$, except in the following cases, in which the
description of $G_{\mathbb R}$ is provided.
\begin{itemize}
\item $n=3$ and there are two possibilities:
\begin{enumerate}
\item[a)] $G_{\mathbb R} = \lambda + \lambda'$ is the union of two distinct lines $\lambda$, $\lambda'$,
both distinct from $\ell_1$; this happens for
$d = 7$, in which case $X=\mathbb P\big(\mathcal O_{\mathbb P^2}(1) \oplus \mathcal O_{\mathbb P^2}(2)\big)$
and $H$ is the tautological line bundle;
\item[b)] $G_{\mathbb R}=2\lambda$ is a double line where $\lambda= \ell_1$, hence
$\Gamma_{\mathbb R} = 3\ell_1$ is a triple line;  this happens for $d = 6$ and either $(X,H) =
\big(\mathbb P^1 \times \mathbb P^1 \times \mathbb P^1, \mathcal O(1,1,1)\big)$, or
$X=\mathbb{P}(T_{\mathbb{P}^2})$ and $H$ is the tautological line bundle.
\end{enumerate}
 \item $n=4$, and $G_{\mathbb R} = \lambda + \lambda'$ is the union of two distinct lines $\lambda$, $\lambda'$,
where, up to renaming, $\lambda = \ell_1$, $\lambda' = \ell_2$, so that
$\Gamma_{\mathbb R} = 2(\ell_1 + \ell_2)$. This happens for $d=6$ and it corresponds
to $(X,H) = \big(\mathbb P^2 \times \mathbb P^2, \mathcal O(1,1)\big)$.
\item $n=5$, and $G_{\mathbb R} = 2\lambda$ with $\lambda = \ell_2$, so that
$\Gamma_{\mathbb R} = \ell_1 + 3\ell_2 + \ell_3$. In this case, $d=5$, and $(X,H)$
is the general hyperplane section of the Grassmannian $\mathbb{G}(1,4)$ embedded in $\mathbb P^9$ via the Pl\"ucker embedding.
\item $n=6$, and $G_{\mathbb R} = \lambda + \lambda'$ is the union of two distinct lines  $\lambda$, $\lambda'$,
 where, up to renamimg, $\lambda = \ell_2$ and $\lambda' = \ell_3$. So
$\Gamma_{\mathbb R} = \ell_1 + 2(\ell_2 + \ell_3) + \ell_4$.
In this case, $d=5$ and $X$ is the Grassmannian $\mathbb{G}(1,4)$ embedded by $H$ in $\mathbb P^9$ via the Pl\"ucker embedding.
\end{itemize}
\end{thm}

\smallskip

As to the situation for $G_{\mathbb Q}$ (when $G_{\mathbb R} \not= \emptyset$), we note the following fact.
First of all, for the term in \eqref{Delta}, we get $\Delta = 0$ when $(n,d) =
(3,6), (5,5)$. Moreover, $\Delta=k^2$ for some $k \in \mathbb Q$ when: $(n,d) =
(3,8), (4,6)$ and $(6,5)$ (in which cases $\Delta = 1$). On the other hand,
$\Delta=k^2$ with $k \not\in \mathbb Q$ if and only if $(n,d)=(3,7)$ (here $\Delta = 4/7$).
Therefore, we have

\begin{prop}\label{2.4}
Let $(X,L)$ be as in case ${\rm (iii)}$ of {\em Proposition \ref{projective space,quadric,del Pezzo}}.
The description of $\Gamma_{\mathbb Q}$ is the same as that
given for $\Gamma_{\mathbb R}$ in {\em Theorem \ref{description}}, up to
regarding $\lambda, \lambda'$ and the $\ell_i$'s as lines in $\mathbb A^2_{\mathbb Q}$,
except when $(n, d)=(3, 7)$, in which case $G_{\mathbb Q}=\emptyset$, so that $\Gamma_{\mathbb Q} = \ell_1$, the
line of equation $ry-2x+1=0$.
\end{prop}

\smallskip

The following table summarizes the above results
about the Hilbert curves
in case (iii) of  Proposition \ref{projective space,quadric,del Pezzo}.

\smallskip

\begin{center}
\footnotesize
\begin{longtable}{cll}
\hline
$n$ & $(X,L)$ & $\Gamma:=\Gamma_{(X,L)}$
\\
\hline
\\
$\geq 3$ &
\begin{tabular}{l}
$(X,H)$ del Pezzo manifold\\
of degree $d:=H^n$,\\
$L=rH$
for $r\geq 1$
\end{tabular} &
\begin{tabular}{l}
$\Gamma = \ell_1+ \dots + \ell_{n-2}+G$ \\
where $\ell_i:\  ry - (n-1)x + i=0$\\
for $i=1,...,n-2$ \ and \\
$G:\ \left( (n-1) \big(x-\frac{1}{2}\big) - ry\right)^2 - \frac{\Delta}{4} =0$,\\
$\Delta= \frac{n-1}{d}\big( (n-1)d-4n\big)$
\end{tabular} \\
& & \\
& Further information on $G_{\mathbb{R}}$ and $G_{\mathbb{Q}}$: &  \\
& & \\
$\geq 7$  & &
\begin{tabular}{l}
$G_{\mathbb{R}}=G_{\mathbb{Q}}=\emptyset$;\ $\Delta<0$, $d\leq 4$
\end{tabular} \\
& & \\
$6$ &
\begin{tabular}{l}
$X=\mathbb{G}(1,4)\subset\mathbb{P}^9$
\end{tabular}
&
\begin{tabular}{l}
$G_{\mathbb{R}}=G_{\mathbb{Q}}=\ell_2+\ell_3$;\  $\Delta=1$, $d=5$
\end{tabular} \\
& & \\
$5$ &
\begin{tabular}{l}
$X=\mathbb{G}(1,4)\cap\mathbb{P}^8$
\end{tabular}
&
\begin{tabular}{l}
$G_{\mathbb{R}}=G_{\mathbb{Q}}=2\ell_2$;\  $\Delta=0$, $d=5$
\end{tabular} \\
& & \\
$4$ &
\begin{tabular}{l}
$X=\mathbb{P}^2\times\mathbb{P}^2$
\end{tabular}
&
\begin{tabular}{l}
$G_{\mathbb{R}}=G_{\mathbb{Q}}=\ell_1+\ell_2$;\  $\Delta=1$, $d=6$
\end{tabular}\\
& & \\
$3$ &
\begin{tabular}{l}
$X=\mathbb{P}^1\times\mathbb{P}^1\times\mathbb{P}^1$, or $\mathbb{P}(T_{\mathbb{P}^2})$
\end{tabular}
&
\begin{tabular}{l}
$G_{\mathbb{R}}=G_{\mathbb{Q}}=2\ell_1$;\  $\Delta=0$, $d=6$
\end{tabular} \\
& & \\

$3$ &
\begin{tabular}{l}
$X=\mathbb{P}\big(\mathcal{O}_{\mathbb{P}^2}(2)\oplus\mathcal{O}_{\mathbb{P}^2}(1)\big)$ \\
\end{tabular}
&
\begin{tabular}{l}
$G_{\mathbb{R}}=\lambda+\lambda', \lambda\neq\lambda'$ and both $\not= \ell_1$ \\
$G_{\mathbb{Q}}=\emptyset$;\ $\Delta =4/7$, $d=7$
\end{tabular} \\
&&\\
\hline
\hline
\\
\caption{Hilbert curves of $(X,L)$ as in ($\diamond$), with $\iota_X= n-1$, $d \leq 7$.}
\label{the geography of HC}
\end{longtable}
\end{center}
This offers the opportunity to point out the role of arithmetic in the study of Hilbert curves.

\begin{rem}\label{analog-conjecture}
{\em Consider the following polarized threefolds:
$(X_1,L_1)= \big(\mathbb P^3, \mathcal O_{\mathbb P^3}(2)\big)$,
$(X_2,L_2)= \big(\mathbb Q^3, \mathcal O_{\mathbb Q^3}(1)\big)$, and the
del Pezzo threefold $(X_2,L_2)$
of degree $7$. Let $p_i=p_{(X_i,L_i)}$ $(i=1,2,3)$ be the polynomials defining the corresponding Hilbert curves.
Then
$$p_1(x,y)=\frac{1}{6} \left(y-2x+\frac{1}{2}\right) (y-2x+1) \left(y-2x+\frac{3}{2}\right) = 0,$$
$$p_2(x,y)= \ \frac{1}{3}(y-3x+1) \left(y-3x+\frac{3}{2}\right)(y-3x+2) = 0,$$
$$p_3(x,y)= \ \frac{7}{6}(y-2x+1) \left((y-2x+1)^2-\frac{1}{7}\right) = 0.$$
Look at them from the real point of view: $(\Gamma_1)_ {\mathbb R}$ consists of three parallel lines
(symmetric with respect to the point $(\frac{1}{2},0)$, with slope $2$, evenly spaced
and with step $\frac{1}{2}$ on the $y$-axis. The shape is the same for $(\Gamma_2)_{\mathbb R}$
except for the slope, which is $3$, and also for $(\Gamma_3)_{\mathbb R}$,
in which case the slope is $2$ again, but here the step on the $y$-axis is $\frac{1}{\sqrt{7}}$,
an irrational number. Clearly, these three curves are equivalent each other from the real affine point of
view, in particular, $\Gamma_1$ and $\Gamma_3$
are even  similar from the Euclidian point of view. However, they are different in terms of their ``geography''
(having either different slopes, or different steps on the $y$-axis).
Moreover, the difference between $\Gamma_1$ and $\Gamma_3$ 
becomes even more evident if we consider their arithmetic, looking at
$(\Gamma_1)_{\mathbb Q}$ and $(\Gamma_3)_{\mathbb Q}$:   
 indeed, while the former consists of three lines, the latter consist of the single line
$y-2x+1=0$, since the second factor of $p_3$ is irreducible over $\mathbb Q$.}
\end{rem}

In line with Proposition
\ref{projective space,quadric,del Pezzo}, we can also obtain a characterization of pairs $(X,L)$ as in $(\diamond)$ with
$\iota_X=n-2$ and $\mathrm{rk}\langle K_X,L\rangle=1$
in terms of Hilbert polynomials, provided that $n\geq 6$ (see \cite[Theorem 3.3]{LT}). The output is the following result.

\begin{prop}\label{Mukai}
Let $(X,L)$ be a polarized manifold of dimension $n\geq 6$.
Then $X$ is a Fano manifold of index $\iota_X=n-2$ and
$L:=\frac{r}{n-2}\left(-K_X\right)$ with $r$ coprime with $n-2$ if and only if
$H=\frac{1}{r}L$ is an ample line bundle and the Hilbert polynomial of $(X,H)$ is as in \eqref{p} with
$$a_0=\frac{1}{(n-3)!},\ a_1=\frac{1}{n!}\left[\left(\frac{d}{2}+1\right)n^2-(2d+1)n+2d\right],\ a_2=\frac{3d}{2\ n!}(n-2),\ a_3=\frac{d}{n!}\ ,$$
where $d:= \left(\frac{-K_X}{n-2}\right)^n$.
\end{prop}

\smallskip

\begin{rem}\label{rem-Mukai}
{\em In Proposition \ref{Mukai} conditions $n\geq 6$ and coprimality
are required only to prove the ``if part''. In fact, a direct check shows that the above
4-tuple $(a_0,a_1,a_2,a_3)$ is the solution of \eqref{linear system}, regardless of the value of $n$, the column vector
on the right hand of \eqref{linear system} being the transpose of
$$\left[ \begin{array}{ c c c c }
\frac{1}{(n-3)!} , & \frac{1}{(n-2)!}\left( n-1+\frac{d}{2}\right) ,
& \frac{2}{(n-1)!}\left( \binom{n}{2}+(n+2)\frac{d}{2}\right) ,
& \frac{6}{n!}\left(\binom{n+1}{3}+
\frac{(n+1)(n+4) d}{4}
\right)
\end{array}\right]\ .$$}
\end{rem}

\medskip

\noindent Let us note that, in principle, Algorithm $1$ in \cite{LT} allows one to
 express $P(z)$ for any pair as in
$(\diamond)$, provided that $h^0(tH)$ is known for every $t=1,\dots,c_X$.

\bigskip

Finally, as to the total reducibility of $P$ over $\mathbb{Q}$ (or $\mathbb{R}$), we have the following result.

\begin{thm}\label{thm Mukai}
Let $(X,L)$ be as in $(\diamond)$ with $\iota_X=n-2>0$.
Let $\Delta:=1-\frac{8n(n-1)(n-2)^{n-2}}{(-K_X)^n}$. Then
$$P(z)=\frac{1}{n!}\frac{(-K_X)^n}{(n-2)^n}\left[z^2+(n-2)z+(n-2)^2\frac{(1-\Delta)}{4}\right]\left(z+\frac{n-2}{2}\right)\prod_{j=1}^{n-3}(z+j)\ ,$$
and the following are equivalent:
\begin{enumerate}
\item[1)] $P$ is totally reducible over $\mathbb{Q}$ (or $\mathbb{R}$);
\item[2)] $\Delta$ is the square of a rational number (or $\Delta\geq 0$).
\end{enumerate}
Moreover, suppose that $n\geq 4$. Then the following are equivalent:
\begin{enumerate}
\item[(I)] $P$ is totally reducible over $\mathbb{Q}$;
\item[(II)] either $b_2(X)=1$ and
$$\big(n,(-K_X)^n\big)\in \{(5,18\cdot 3^5),(6,16\cdot 4^6),(7,14\cdot 5^7),(8,14\cdot 6^8),\\ (9,12\cdot 7^9),(10,12\cdot 8^{10})\},$$
or $b_2(X)\geq 2$ and
$X$ is one of the following Mukai manifolds (where $T_{\mathbb P^m}$ stands for the tangent bundle to $\mathbb P^m$):
\begin{enumerate}
\item $\mathbb{P}^3\times\mathbb{P}^3$;
\item $\mathbb{P}^2\times\mathbb{Q}^3$;
\item $\mathbb{P}(T_{\mathbb{P}^3})$;
\item $\mathbb{P}^1\times\mathbb{P}^3$;
\item $\mathbb{P}^1\times\mathbb{P}(T_{\mathbb{P}^2})$;
\item $\mathbb{P}^1\times\mathbb{P}^1\times\mathbb{P}^1\times\mathbb{P}^1$;
\item $\mathbb{P}(\mathcal{N})$, where $\mathcal{N}$ is the null-correlation bundle on $\mathbb{P}^3$;
\end{enumerate}
\end{enumerate}
\end{thm}

For the description of $X$ in the case $b_2(X)=1$ we refer to \cite[Theorem 5.2.3 and Examples 5.2.2, (ii)--(v)]{IP}.

\begin{proof}
From Proposition \ref{basic} and  Proposition \ref{thm1}, it follows that
$$P(z)=\frac{1}{n!}\frac{(-K_X)^n}{\iota_X^n}(z^2+\alpha z+\beta)\left(z+\frac{n-2}{2}\right)
\prod_{j=1}^{n-3}(z+j)\ ,$$
where
$\alpha =(n-2)$ and
$\beta =(n-2)^2\frac{(1-\Delta)}{4}$
in view of Proposition \ref{Mukai} and Remark \ref{rem-Mukai}.
The equivalence between 1) and 2) follows immediately from the fact that the discriminant of $z^2+\alpha z+\beta$
is equal to $(n-2)^2\Delta$.
Assume now that $n\geq 4$ and let $b_2(X)\geq 2$. Then from \cite[Theorem 7.2.1]{IP} and
\cite{W2} we know that $n-2=\iota_X\leq \frac{n}{2}+1$, i.e. $n\leq 6$. If $n=6$, then $\iota_X=\frac{n}{2}+1$ and by
\cite[Theorem 7.2.2 (i)]{IP} and \cite{W2,W3} we deduce that $X\cong\mathbb{P}^3\times\mathbb{P}^3$, hence $P$
is totally reducible over $\mathbb{Q}$ by \cite[p.\ 466]{BLS}.
If $n=5$, then $\iota_X=\frac{n+1}{2}$ and from
\cite[Theorem 7.2.2 (ii)]{IP} and \cite{W2,W3} it follows that either
$(-K_X)^5=20\cdot 3^5$ and $X$ is $\mathbb{Q}^3\times\mathbb{P}^2$ or $\mathbb{P}(T_{\mathbb{P}^3})$,
and in these cases $P$ is totally reducible over $\mathbb{Q}$, since condition 2) is fulfilled, or
$(-K_X)^5= 26\cdot 3^5$,
$X=\mathbb{P}\big(\mathcal{O}_{\mathbb{P}^3}(1)^{\oplus 2}\oplus\mathcal{O}_{\mathbb{P}^3}(2)\big)$,
and in this case $P$ is not totally reducible over $\mathbb{Q}$ by 2) again.
Suppose that $n=4$, hence $\iota_X=2$. Then
$\Delta = 1-\frac{384}{(-K_X)^4}=\frac{H^4-24}{H^4}$ and by a close inspection of
\cite[Theorem 7.2.15 and table 12.7 at p.\ 225]{IP}
we see that $X$ is as in (d)--(g).
Now let $b_2(X)=1$. If $P$ is totally reducible over $\mathbb Q$,
by the equivalence between 1) and 2) it follows that $\Delta\geq 0$,
i.e. $(-K_X)^n\geq 8n(n-1)(n-2)^{n-2}$. Moreover, by the
genus formula we have $(-K_X)^n=(2g-2)(n-2)^n$, where $g:=g\left(X,\frac{1}{n-2}(-K_X)\right)$
is the sectional genus of the polarized manifold
$\left(X,\frac{1}{n-2}(-K_X)\right)$. This implies that
$$g\geq \frac{4n(n-1)}{(n-2)^2}+1=5+\frac{12n-16}{(n-2)^2}\ .$$ Since $n\geq 4$, we get $g\geq 6$ and
by the above inequality for $g$ and \cite{Mu} and \cite{Mu1}, combined with \cite{Me} (see also \cite[Theorem 5.2.3]{IP}),
we obtain that $(g,n)\in\{(10,5),(9,6),(8,7),(8,8),(7,9),$
$ (7,10)\}$. This gives the pairs $\left(n,(-K_X)^n\right)$ as in (II) for the case $b_2(X)=1$.
Therefore
(I) implies (II).
To prove the converse, it is enough
to compute $\Delta$ and check condition 2),
 in view of the first part of the statement.
 \end{proof}

\section{Low-dimensional toric Fano manifolds}\label{Sec3}

In this section we analyze the total reducibility of $P$ over $\mathbb Q$ (and $\mathbb R$) for pairs $(X,L)$
as in ($\diamond$) when $X$ is a toric Fano $n$-fold with $n\leq 4$. First of all recall that for $n=2$
the toric Fano manifolds are $\mathbb P^1 \times \mathbb P^1$ and $\mathbb P^2$ blown-up at $s \leq 3$ fixed points of the
torus action. Hence those with $P$ totally reducible over $\mathbb Q$ are exactly the surfaces listed in Proposition \ref{surface}.
We can thus assume that $n \geq 3$.

Assume that $(X,L)$ is as in $(\diamond)$. If $\iota_X\geq n$, then by the
Kobayashi--Ochiai Theorem \cite[Corollary 3.1.15]{IP} we know that
$(X,H)\cong \big(\mathbb{P}^n,\mathcal{O}_{\mathbb{P}^n}(1)\big)$,
$\big(\mathbb{Q}^n,\mathcal{O}_{\mathbb{Q}^n}(1)\big)$
with $L=rH$ and $\iota_X=n+1,n$, respectively. Note that $\mathbb Q^n$
is not toric
for $n\geq 3$.
Furthermore, if $\iota_X=n-1$ then we can rely on Theorem \ref{description}.
In fact, the only toric
del Pezzo $n$-folds are those with $n \leq 4$ in Theorem \ref{description}, because $\mathbb P^n$ is the only toric
(Fano) $n$-fold with Picard number $1$.

\smallskip

Now we address cases $n=3$ and $n=4$.

\smallskip

First, let $X$ be any $n$-fold and write a divisor $D$ on $X$ as $D= \frac{1}{2}K_X + E$. If $n=3$, then the Riemann--Roch--Hirzebruch formula can be written in the following form (see \cite[(7)]{BLL}):
\begin{equation}\label{RR3}
\chi(D) = \frac{1}{6}E^3 + \frac{1}{24} \big(2c_2(X)-K_X^2\big) \cdot E\ .
\end{equation}
Similarly, if $n=4$, we can write (see \cite[(8)]{BLL}):
\begin{equation}\label{RR4}
\chi(D) = \frac{1}{24}E^4 + \frac{1}{48}\big(2c_2(X)-K_X^2\big) \cdot E^2  +
\frac{1}{384} \big(K_X^2 - 4c_2(X)\big)\cdot K_X^2+\chi(\mathcal O_X)\ .
\end{equation}

Next, let $(X,L)$ be as in ($\diamond$) and set $D=xK_X+yL$. Note that writing $z:= ry-\iota_Xx$
and $w:=z+\frac{\iota_X}{2}$, as in Section \ref{backgr}, we have
\begin{equation}\label{E}
E =\left(x-\frac{1}{2}\right)K_X+yL =\left( ry-\iota_X x+\frac{\iota_X}{2}\right)H=\left(z+\frac{\iota_X}{2}\right)H = wH= \frac{w}{\iota_X}(-K_X).
\end{equation}
Thus, for $n=3$ we have the following result.

\begin{thm}\label{propo1}
Let $(X,L)$ be as in $(\diamond)$ with $n=3$. Then
$P$ is totally reducible over $\mathbb{R}$ if and only if $48\leq (-K_X)^3\leq 64$. Moreover,
the following are equivalent:
\begin{enumerate}
\item[1)] $P$ is totally reducible over $\mathbb{Q}$;
\item[2)] $(-K_X)^3\in\{48,50,54,64\}$;
\item[3)] $X$ is one of the following Fano threefolds:
\begin{enumerate}
\item $\mathbb{P}^3$;
\item $\mathbb{Q}^3$;
\item $\mathbb{P}(T_{\mathbb{P}^2})$, where $T_{\mathbb{P}^2}$ is the tangent bundle to $\mathbb{P}^2$;
\item $\mathbb{P}^1\times\mathbb{P}^2$;
\item $\mathbb{P}\big(\mathcal O_{\mathbb{P}^1}^{\oplus 2} \oplus  \mathcal O_{\mathbb{P}^1}(1)\big)$;
\item $\mathbb{P}^1\times\mathbb{P}^1\times\mathbb{P}^1$;
\item $\mathbb{P}^1\times\mathbb{F}_1$;
\item $X$ is the blow-up of $V_7\cong \mathbb{P}\big(\mathcal{O}_{\mathbb{P}^2}\oplus\mathcal{O}_{\mathbb{P}^2}(1)\big)$ along a line lying on the exceptional divisor of the blow-up $V_7\to\mathbb{P}^3$;
\item $X$ is the blow-up of $V_7\cong \mathbb{P}\big(\mathcal{O}_{\mathbb{P}^2}\oplus\mathcal{O}_{\mathbb{P}^2}(1)\big)$ along the proper transform of a line passing through the center of the blow-up $V_7\to\mathbb{P}^3$.
\end{enumerate}
\end{enumerate}
\end{thm}

\begin{proof}
Recall that $c_2(X) \cdot K_X= -24\chi(\mathcal O_X) = - 24$ and $(-K_X)^3 \leq 64$ for any Fano 3-fold $X$.
So, from \eqref{RR3} and \eqref{E} we get
$$ Q(w) = \frac{1}{24\ \iota_X^3}\ w \Big(4(-K_X)^3w^2 - \iota_X^2\big((-K_X)^3-48\big) \Big).$$
Therefore, we deduce that $P$ is totally reducible over $\mathbb R$ if and only if  $48\leq (-K_X)^3 \leq 64$,
which gives the first part of the statement. Furthermore,

\smallskip

\begin{equation}\label{condition}
P\ \textrm{ is totally reducible over}\  \mathbb Q\ \textrm{ if and only if }\
1 - \frac{48}{(-K_X)^3}\ \textrm{ is the square of a rational number}\ .
\end{equation}

\smallskip

\noindent
Then the equivalence between $1), 2)$ and $3)$ follows from \eqref{condition} and the classification of Fano $3$-folds (see \cite[pp.\ 214--224]{IP} and \cite{MM}).
\end{proof}

Comparing Theorem \ref{propo1} with the table in \cite[Remark 2.5.10]{Ba}, we obtain the following immediate consequence for toric Fano $3$-folds.

\begin{cor}\label{cor}
Let $(X,L)$ be as in $(\diamond)$ and assume that $X$ is a toric Fano
$3$-fold. Then $P$
is totally reducible over $\mathbb{Q}$ if and only if
$X$ is one of the
Fano $3$-folds listed in $3)$ of Theorem {\rm\ref{propo1}}, except cases {\rm (b)} and {\rm (c)}.
\end{cor}

Finally, we obtain the following result for toric Fano $4$-folds.

\begin{thm}\label{prop1}
Let $(X,L)$ be as in $(\diamond)$ and assume that $X$
is a toric Fano $4$-fold. Let $k:=K_X^4$ and $h:=c_2(X)\cdot K_X^2$. Then $P$ is totally reducible over $\mathbb{R}$
if and only if $X$ is as in \cite[Table in $\S 4$]{Ba} with $h^2\geq 96k$ and $k\geq 2h+2\sqrt{h^2-96k}$.
Moreover, the following are equivalent:
\begin{enumerate}
\item[$1)$] $P$ is totally reducible over $\mathbb{Q}$;
\item[$2)$] one of the following cases occurs:
\begin{longtable}{|l|c|c|c|c|}
\hline
N. &   N. in \cite{Ba} & $k$   & $h$  &  Description of $X$ in \cite[\S 4]{Ba}
\\
\hline
$1$ & $1$ & $625$ & $250$ &   $\mathbb{P}^4$ \\
\hline
$2$ & $5$ & $512$ & $224$ &   $\mathbb{P}^3\times\mathbb{P}^1$ \\
$3$ & $6$ &  &   &    $\mathbb{P}_{\mathbb{P}^1}(\mathcal{O}^{\oplus 3}\oplus\mathcal{O}(1))$ \\
\hline
$4$ & $10$  & $486$ & $216$  & $\mathbb{P}^2\times\mathbb{P}^2$ \\
$5$ & $20$ &  &  &   $\mathbb{P}_{\mathbb{P}^1\times\mathbb{P}^1}(\mathcal{O}^{\oplus 2}\oplus\mathcal{O}(1,1))$ \\
\hline
$6$ & $21$  & $432$ & $204$ &  $\mathbb{P}^1$-bundle over $\mathbb{P}_{\mathbb{P}^2}(\mathcal{O}\oplus\mathcal{O}(1))$ \\
$7$ & $26$ &  &   &   $\mathbb{P}^2\times\mathbb{P}^1\times\mathbb{P}^1$ \\
$8$ & $27$ &  &  &   $\mathbb{P}^1\times\mathbb{P}_{\mathbb{P}^1}(\mathcal{O}^{\oplus 2}\oplus\mathcal{O}(1))$ \\
$9$ & $28$ &  &  &   $\mathbb{P}^2\times\mathbb{F}_1$ \\
$10$ & $29$ &  &  &  $\mathbb{P}^1$-bundle over $\mathbb{P}_{\mathbb{P}^2}(\mathcal{O}\oplus\mathcal{O}(2))$ \\
$11$ & $52$ &  &  &  $\mathbb{P}^1$-bundle over $\mathbb{P}^1\times\mathbb{P}^1\times\mathbb{P}^1$ \\
\hline
$12$ & $31$ & $400$ & $196$ &  $\mathbb{P}_{\mathbb{P}^2\times\mathbb{P}^1}(\mathcal{O}\oplus\mathcal{O}(2,-1))$ \\
$13$ & $32$ &  &  &   $\mathbb{P}_{\mathbb{P}^2\times\mathbb{P}^1}(\mathcal{O}\oplus\mathcal{O}(1,-1))$ \\
$14$ & $54$ &  &  &   $\mathbb{P}_{\mathbb{P}^1\times\mathbb{F}_1}(\mathcal{O}\oplus\mathcal{O}(l))$, $l$ is a curve of
index $1$ on $\mathbb{F}_1$ \\
$15$ & $68$ &  &  &   $\mathbb{P}^1\times W$, $W$ is the blow-up of $\mathbb{P}^1$ on $\mathbb{P}_{\mathbb{P}^2}(\mathcal{O}\oplus\mathcal{O}(1))$ \\
\hline
$16$ & $55$ & $384$ & $192$ &   $\mathbb{F}_1\times\mathbb{F}_1$ \\
$17$ & $56$ &  &  &  $\mathbb{P}^1\times\mathbb{P}^1\times\mathbb{P}^1\times\mathbb{P}^1$ \\
$18$ & $57$ &  &  &  $\mathbb{P}^1\times\mathbb{P}^1\times\mathbb{F}_1$ \\
$19$ & $58$ &  &  &  $\mathbb{P}^1$-bundle over the blow-up of $\mathbb{P}^1$ on $\mathbb{P}^2\times\mathbb{P}^1$ \\
$20$ & $69$ &  &  &  the blow-up of $\mathbb{P}^1\times\mathbb{P}^1$ on N. $6$ \\
$21$ & $72$ &  &  &  $\mathbb{P}^1$-bundle over the blow-up of $\mathbb{P}^1$ on $\mathbb{P}_{\mathbb{P}^2}(\mathcal{O}\oplus\mathcal{O}(1))$ \\
\hline
\end{longtable}
\end{enumerate}
\end{thm}

\begin{proof}
Noting that $\chi(\mathcal O_X)=1$ since $X$ is Fano, and recalling that $k=K_X^4$ and $h=c_2(X)\cdot K_X^2$,
from \eqref{RR4} and \eqref{E} we obtain that
$$Q(w)= \frac{1}{384\ \iota_X^4}\ \big(16 k w^4 + 8\iota_X^2(2h-k)w^2 + \iota_X^4(k-4h+384)\big)\ ,$$
where $w:=z+\frac{\iota_X}{2}$.
Let $\Delta$ be the discriminant of the above biquadratic trinomial in brakets. Then
$$\frac{\Delta}{4} := 16\iota_X^4 \big( (2h-k)^2 -k (k-4h+384) \big) = 64\iota_X^4 (h^2-96k).$$
Therefore,
\begin{equation*}
\qquad P\ \textrm{ is totally reducible over}\  \mathbb Q\ (\mathbb R)\ \textrm{ if and only if }\
\textrm{ there exist}\ \alpha,\beta,\gamma\in\mathbb{Q}\ (\mathbb{R})\textrm{ such that }
\end{equation*}
\begin{equation}\label{condition bis}
h^2-96k=\alpha^2,\quad \ \frac{k-2h+2\alpha}{k}=\beta^2\quad \ \textrm{ and }\quad \ \frac{k-2h-2\alpha}{k}=\gamma^2\ .
\end{equation}
Observe that the first part of the statement simply follows from the first and the third
conditions in \eqref{condition bis} (assuming that $\alpha\geq 0$). Finally, the equivalence between $1)$ and $2)$ follows directly from \cite[Table in $\S 4$]{Ba}, keeping in mind all the three conditions in \eqref{condition bis}. To simplify the check, observe that
for some particular Fano 4-folds $X$ in the table we already know that $P$ is totally reducible over $\mathbb Q$: this is true for $N.$ 1, by Proposition \ref{projective space,quadric,del Pezzo} (i), and for $N.$  2, 4, 7, and 17
because in these cases $P$ is the product of polynomials which are totally reducible over $\mathbb Q$ \cite[p.\ 466]{BLS}.
Note also that cases corresponding to the same values of $k$ and $h$ in the table are grouped by horizontal lines and they have the same $\alpha$. It thus follows that,
for all $X$ belonging to a group containing one of the above cases, the three conditions in \eqref{condition bis} are
trivially satisfied, hence the corresponding $P$ is totally reducible over $\mathbb Q$.
Therefore the only group for which one needs to check \eqref{condition bis} is that containing $N.$ 12--15. In this case
we have $h^2-96k=4^2$, hence $\alpha=4$; on the other hand $\beta ^2 = \frac{1}{400} 4\alpha = \left(\frac{1}{5}\right)^2$ and $k-2h=8=2\alpha$, hence $\beta =\frac{1}{5},\gamma=0$.
\end{proof}

\medskip

\begin{rem}
{\em With a similar approach, by using computer algebra programs  and relying on the database at}

\medskip

{\em
\url{http://www.grdb.co.uk/forms/toricsmooth?dimension_cmp=eq&dimension=3}\ ,
}

\medskip

\noindent {\em one could provide results in line with Theorem \ref{prop1} also for $n=5,6$,
with a case-by-case analysis; the lists, however, become very long.}
\end{rem}

\section{Fano bundles of large index}\label{Sec4}

Let $X=\mathbb P_Y(\mathcal E)$, where $\mathcal E$ is an ample vector bundle
on a projective manifold $Y$ of dimension $m \geq 2$,
such that $K_Y+\det \mathcal E= \mathcal O_Y$. Bundles like $X$ are called Fano bundles or ruled Fano
manifolds (or equivalently $(Y,\mathcal E)$
is said to be a Mukai pair, according to \cite{K}), since $X$ is Fano; actually, by the canonical bundle formula
$-K_X = {\rm rk}(\mathcal E) \xi$, $\xi$ being the
tautological line bundle, which is ample, so being $\mathcal E$. Moreover, since $\xi \cdot \ell =1$ for any line $\ell$ contained in a fiber of the bundle projection
$X \to Y$, we get $\iota_X = {\rm rk}(\mathcal E)$, hence $H=\xi$.
For ${\rm rk}(\mathcal E) \geq m-2$, Fano bundles are (almost) completely classified.
This comes from work of Fujita, Ye and Zhang, and Peternell for ${\rm rk}(\mathcal E) \geq m$
(see references cited in \cite{PSW}); for ${\rm rk}(\mathcal E) = m-1$, we refer to
\cite{W} for $m=3$,  \cite[Theorem 0.3]{PSW} for $m \geq 5$ and
\cite[Proposition 7.4]{PSW} combined with \cite{O} (for the elimination of a doubtful case in \cite{PSW}) for $m=4$. Case ${\rm rk}(\mathcal E) = m-2$ has been recently studied
by Kanemitsu \cite{K} for $m\geq 5$; for $m=4$, see \cite{NO}.
In connection with our problem, here we study the Hilbert polynomial of these Fano bundles,
when ${\rm rk} (\mathcal E)\geq m$, in which case their classification is complete. Notice that, according to results of Wi\'sniewski \cite{W2}, \cite{W3}, they exhaust the class
of Fano manifolds $X$ of index $\iota_X \geq (\dim X +1)/2$ and second Betti number $b_2(X)\geq 2$.

\smallskip

To begin our analysis, we recall that
if ${\rm rk} (\mathcal E) > m$ there is only one possible pair as above, namely $\big(\mathbb P^m, \mathcal O_{\mathbb P^m}(1)^{\oplus (m+1)}\big)$
\cite[Main Theorem]{Fu1}.
In this case $n = \dim X = 2m$, and
$(X, H) = \big(\mathbb P^m \times \mathbb P^m, \mathcal O_{\mathbb P^m \times \mathbb P^m}(1,1)\big)$,
in particular $\iota_X = m+1$ and $L = \big(\mathcal O_{\mathbb P^m}(r)\big)^{\boxtimes 2}$.
Since $\chi\big(\mathcal O_{\mathbb P^m \times \mathbb P^m}(k,k)\big) =\big(\chi\big(\mathcal O_{\mathbb P^m}(k)\big)\big)^2$
for every $k$, we have
$$P(z) =  \left(\frac{1}{m!} \prod_{j=1}^m( z+j)\right)^2,$$
by Proposition \ref{projective space,quadric,del Pezzo} (i). Hence $P$ is totally reducible over $\mathbb Q$.

Next suppose that ${\rm rk}(\mathcal E) = m$, so that $n = \dim(X) = 2m-1$, and $\iota_X = c_X = m$. In this case, $(Y,\mathcal E)$ is one of the
following pairs (see \cite[Theorem 0.1]{PSW}):
\begin{enumerate}
\item[1)] $(\mathbb P^m, T_{\mathbb P^m})$, where $T_{\mathbb P^m}$ denotes the tangent bundle,
\item[2)] $\big(\mathbb P^m, \mathcal O_{\mathbb P^m}(1)^{\oplus (m-1)} \oplus \mathcal O_{\mathbb P^m}(2)\big)$,
\item[3)] $\big(\mathbb Q^m, \mathcal O_{\mathbb Q^m}(1)^{\oplus m}\big)$.
\end{enumerate}

For any pair in the list above and for any positive integer $r$,
consider the polarized manifold $(X=\mathbb P_Y(\mathcal E), L=rH)$, where
$H= \xi$, the tautological line bundle on $X$, mentioned above. Clearly $(X,L)$ is as in ($\diamond$).
To make the polynomial $P$ explicit in the above cases we should determine
the coefficients $a_j$ in the factor $R(x,y)$ appearing in Proposition \ref{basic}. This requires
to compute, for any integer $t \leq m$, $h^0(tH)$, which is equal to $h^0(S^t \mathcal E)$,
$S^t$ standing for the $t$-th symmetric power,
and $\delta(t) = (1+t)(2+t) \dots (t+m-1)$.

\smallskip

First of all, let us consider case 2).
In this case we already know that
$P$ is not totally reducible over $\mathbb Q$ for $m=2$ (Proposition \ref{basic2} iv); see also Remark \ref{analog-conjecture})
and $m=3$ (Theorem \ref{thm Mukai}).
On the other hand, we can check that this fact is true at least for small values of $m$ by a direct
computation along the line described above.

As a first thing let us compute the vector appearing on the right hand of \eqref{linear system} in general.
To do that we need to determine $h^0(tH)$ for any $t=0, 1, \dots, m=c_X$. Write
$\mathcal E = \mathcal F \oplus \mathcal O_{\mathbb P^m}(2) $,
where $\mathcal F = \mathcal O_{\mathbb P^m}(1)^{\oplus (m-1)}$.
Note that
$$ S^k \mathcal F = \mathcal O_{\mathbb P^m}(k)^{\oplus{\binom{k+m-2}{m-2}}}, \qquad \text{\rm{for every}} \ k \geq 0.$$
Then
$$S^t \mathcal E = \bigoplus_{j=0}^t \Big( S^j \big( \mathcal O_{\mathbb P^m}(2) \big) \otimes S^{t-j} \mathcal F \Big) = \bigoplus_{j=0}^t \Big( \mathcal O_{\mathbb P^m}(t+j)^{\oplus{\binom{t-j+m-2}{m-2}}} \Big)\ .$$
Therefore,
$$h^0(tH)=h^0(S^t\mathcal E) = \sum_{j=0}^t \binom{t-j+m-2}{m-2} h^0\big(\mathcal O_{\mathbb P^m}(t+j)\big)=
\sum_{j=0}^t \binom{t-j+m-2}{m-2} \binom{t+j+m}{m}.$$
In conclusion, recalling the expression of $\delta (t)$ in Proposition \ref{basic}, we see that
the $(t+1)$-th component of the column vector on the right hand of \eqref{linear system} is
$$ \frac{t!}{(t+m-1)!} \sum_{j=0}^t \binom{t-j+m-2}{m-2} \binom{t+j+m}{m}$$
for $t=0, 1, \dots, m$.

For example, let $m=2$; then $\delta(t) = t+1$, moreover,
$h^0(H)=6+3=9$ and $h^0(2H)= 15+10+6=31$. Thus the column vector on the right hand of \eqref{linear system}
is the transpose of
$[1 \quad \frac{9}{2} \quad \frac{31}{3}]$. Since in the present case the inverse of the matrix in \eqref{U} is
\begin{equation}\label{U2inverse}
U^{-1}=\left( \begin{array}{ c c c }
1& 0 & 0 \\
-\frac{3}{2} & 2 & -\frac{1}{2} \\
\frac{1}{2} & -1 & \frac{1}{2} \\
\end{array}\right)\ ,
\end{equation}
we get
$[a_0 \quad a_1 \quad a_2] = \big[1 \quad \frac{7}{3} \quad \frac{7}{6}\ \big]$,
hence the first factor of $P(z)$ in \eqref{p} is $\frac{1}{6} (7z^2 + 14z + 6)$.
Note that this trinomial is exactly that appearing in
Proposition \ref{projective space,quadric,del Pezzo} (iii), before the product, in the situation corresponding to the
case at hand. Moreover,  it is clearly not totally reducible over $\mathbb{Q}$.
The same happens for $m=3$, as shown in the proof of Theorem \ref{thm Mukai}. In this case
the first factor of $P(z)$ in \eqref{p} is
$\frac{1}{120}(26z^3 + 117z^2 + 157z + 60)$, which
is totally reducible over $\mathbb R$ but not over $\mathbb Q$.
Furthermore, for $4\leq m\leq 150$, one can use the following Magma Program \cite{magma}:

\smallskip
\begin{verbatim}
K:=Rationals();
P<x>:=PolynomialRing(K,1);
V:=function(m);
L:=[];
 for k in [1..m+1] do
  L:=L cat [1];
 end for;
 for j in [1..m] do
  L:=L cat [i^j : i in [0..m]];
 end for;
return Matrix(K,m+1,m+1,L);
end function;
v:=function(m);
L:=[];
 for t in [0..m] do
  L:= L cat [(Factorial(t)/Factorial(t+m-1))*
  (&+[Binomial(t-j+m-2,m-2)*Binomial(t+j+m,m): j in [0..t]])];
 end for;
return Matrix(K,1,m+1,L);
end function;
Test:=function(N,NN);
M:={};
 for m in [N..NN] do
  m; a:=v(m)*V(m)^(-1); p:=P!&+[a[1,j+1]*x^(j) : j in [0..m]];
  b:=&+[Factorization(p)[k][2]: k in [1..#Factorization(p)]];
  if b eq m then
   M:=M join {<m, p>};
  end if;
 end for;
return M;
end function;
\end{verbatim}

\noindent Typing

\smallskip
\begin{verbatim}
Test(4,150);
\end{verbatim}

\smallskip
\noindent in the Magma calculator, one can check that for $4 \leq m \leq 150$
the factor of $P(z)$ corresponding to the polynomial $R(x,y)$ is always not totally reducible over $\mathbb{Q}$.

\smallskip

Next, let us settle case 3).
Since $(Y, \mathcal E) = \big(\mathbb Q^m, \mathcal O_{\mathbb Q^m}(1)^{\oplus m}\big)$, we have that
$(X,H)=\big(\mathbb Q^m \times \mathbb P^{m-1}, \mathcal O_{\mathbb Q^m \times \mathbb P^{m-1}}(1,1)\big)$, hence
$L = \mathcal O_{\mathbb Q^m}(r) \boxtimes \mathcal O_{\mathbb P^{m-1}}(r)$.
Then $\chi\big(\mathcal O_{\mathbb Q^m \times \mathbb P^{m-1}}(k,k)\big)=
\chi\big(\mathcal O_{\mathbb Q^m}(k)\big)  \chi\big(\mathcal O_{\mathbb P^{m-1}}(k)\big)$ for every $k$.
So, according to Proposition \ref{projective space,quadric,del Pezzo} (i) and (ii), we get
$$P(z) = \frac{2}{m! (m-1)!}\ \left(z+\frac{m}{2}\right) \left(\prod_{j=1}^{m-1}(z+j)\right)^2.$$
Therefore, in case 3) $P$ is totally reducible over $\mathbb Q$.

\smallskip
In case 1) we have already seen that $P$ is totally reducible over $\mathbb Q$
for $m=2$ (Theorem \ref{description} and Proposition \ref{2.4}; see also Table \ref{the geography of HC})
and $m=3$ (Theorem \ref{thm Mukai}).
To deal with case 1) in general,  note that our pair $(X,H) = \big(\mathbb P(T_{\mathbb P^m}), \xi\big)$
($\xi$ being the tautological line bundle) is the general hyperplane section of
$\big(\mathbb P^m \times \mathbb P^m, \mathcal O_{\mathbb P^m \times \mathbb P^m}(1,1)\big)$. So denoting this last pair
by $(\mathcal X, \mathcal H)$,
we have that $X \in |\mathcal H|$ and $H=\mathcal H_X$. Hence we can use the following
\begin{lem} \label{lemma}
Let $(\mathcal X, \mathcal H)$ be a
polarized manifold such that $|\mathcal H|$ contains a smooth element $X$ and let
$H=\mathcal H_X$. Set $L=rH$, and $\mathcal L=r\mathcal H$ for some positive integer $r$. Then
$$p_{(X,L)}(x,y) = p_{(\mathcal X, \mathcal L)}\left(x, \frac{x}{r}+y\right) -
p_{(\mathcal X, \mathcal L)}\left(x, \frac{x-1}{r}+y\right).$$
\end{lem}
\begin{proof} We have $xK_X+yL = \big(x(K_{\mathcal X}+ \mathcal H)+ry \mathcal H\big)_X = \big(xK_{\mathcal X} + (x+ry)\mathcal H \big)_X$, by adjunction.
Tensoring the exact sequence
$$0 \to \mathcal O_{\mathcal X}(-\mathcal H) \to \mathcal O_{\mathcal X} \to \mathcal O_X \to 0$$
by $xK_{\mathcal X}+ (\frac{x}{r}+y)\mathcal L$ we thus get
$$0 \to \mathcal O_{\mathcal X}\left(x K_{\mathcal X} + \big(\frac{x}{r}+y-\frac{1}{r}\big)\mathcal L\right) \to \mathcal O_{\mathcal X}\left(x K_{\mathcal X} + \big(\frac{x}{r}+y\big)\mathcal L\right)
\to \mathcal O_X(xK_X+yL) \to 0.$$
Therefore,
$$\chi(xK_X+yL) = \chi \left(x K_{\mathcal X} + \big(\frac{x}{r}+y\big)\mathcal L \right) -
\chi \left(x K_{\mathcal X} + \big(\frac{x}{r}+y-\frac{1}{r}\big)\mathcal L\right),$$
which proves the assertion.
\end{proof}
Now apply  Lemma \ref{lemma} to case 1). We have
$$p_{(\mathcal X, \mathcal L)}\left(x, \frac{x}{r}+y\right) =
\left(\frac{1}{m!}\right)^2 \left(\prod_{j=1}^m(ry-mx+j)\right)^2$$
and
$$p_{(\mathcal X, \mathcal L)}\left(x, \frac{x-1}{r}+y\right) = \left(\frac{1}{m!}\right)^2
\left(\prod_{j=1}^m(ry-mx+j-1)\right)^2.$$
Set $z:=ry-mx$. Then
\begin{eqnarray*}
P(z) & = & \left(\frac{1}{m!}\right)^2 \left[\left(\prod_{j=1}^m (z+j)\right)^2 - \left(\prod_{j=1}^m (z+j-1)\right)^2 \right] \\
& = & \left(\frac{1}{m!}\right)^2\ \left[\prod_{j=1}^m (z+j) - \prod_{j=1}^m (z+j-1)\right] \cdot \left[\prod_{j=1}^m(z+j) + \prod_{j=1}^m(z+j-1)\right] \\
& = & \left(\frac{1}{m!}\right)^2\ \left[m \prod_{j=1}^{m-1} (z+j) \right] \cdot \left[(2z+m) \prod_{j=1}^{m-1} (z+j)\right] \\
& = & \left(\frac{1}{m!}\right)^2\ m\ (2z+m) \left(\prod_{j=1}^{m-1} (z+j)\right)^2. \\
\end{eqnarray*}
In conclusion, $P(z)$ has the same expression as in case 3). Therefore
$P$ is totally reducible over $\mathbb Q$ even in case 1) of the list.

\smallskip

The above discussion proves the following results.
\begin{thm} \label{final}
Let $(X,L)$ be as in $(\diamond)$ with
$X$ being a Fano bundle of index $\iota_X \geq \frac{n+1}{2}$. Suppose that either $n\geq 4$ is even, or
$n=2m-1 \geq 3$ and $X\ncong\mathbb{P}\left( \mathcal{O}_{\mathbb{P}^m}(1)^{\oplus{m-1}}
\oplus\mathcal{O}_{\mathbb P^m}(2)\right)$.
Then the polynomial $P$ is totally reducible over $\mathbb Q$.
\end{thm}

\begin{prop} \label{final bis}
Let $(X,L)$ be as in $(\diamond)$ with $n\leq 299$ and $X$ being a Fano bundle of index $\iota_X \geq \frac{n+1}{2}$. Then the following are equivalent:
\begin{enumerate}
\item[(i)] the polynomial $P$
is totally reducible over $\mathbb Q$;
\item[(ii)] $X\ncong\mathbb{P}\left(\mathcal{O}_{\mathbb{P}^{(n+1)/2}}(1)^{\oplus (n-1)/2}
\oplus\mathcal{O}_{\mathbb P^{(n+1)/2}}(2)\right)$.
\end{enumerate}
\end{prop}

\smallskip

\noindent Finally, in line with Proposition \ref{final bis}, let us state here the following conjecture.

\smallskip

{\bf Conjecture} 1. Let $(X,L)$ be as in $(\diamond)$, $X=\mathbb P(\mathcal E)$ being an
$n$-dimensional Fano bundle over $Y$, with  rk$(\mathcal E) \geq \dim Y$
(or equivalently, $\iota_X \geq \frac{n+1}{2}$). Then the polynomial $P$
is totally reducible over $\mathbb Q$ if and only if $(Y,\mathcal E)$ is not as in case $2)$.

\smallskip

One more remark concerning case $2)$. Relying on computational experiments done with Magma
for low values of $m$ we formulate also the following conjecture.

\smallskip

{\bf Conjecture} 2. For $(X,L)$ as in case $2)$, the factor of $P(z)$ corresponding to
$R(x,y)$ in Proposition \ref{basic} is totally reducible over $\mathbb R$ for any $m$, but it has
either no rational zero or a single rational zero, namely $- \frac{m}{2}$,
according to whether $m$ is either even or odd, respectively.

\medskip

\noindent {\bf Acknowledgements}. The first author is a member of G.N.S.A.G.A. of the Italian INdAM. He would like to thank the PRIN 2015 Geometry of Algebraic Varieties and the University of Milano for partial support.
 During the preparation of this paper, the second author was partially supported by the Proyecto VRID N.219.015.023-INV of the University of Concepci\'on.

\end{document}